\documentclass[10pt]{article}

\usepackage{amsfonts,a4wide,amsmath,amssymb,amsthm}
\usepackage[utf8]{inputenc}  
\usepackage[T1]{fontenc}
\usepackage[all]{xy}    
\usepackage{mathtools}  
\usepackage{multirow}    
\usepackage{csquotes}	
\usepackage{bm} 
\usepackage{tikz}   
\usepackage{float}
\usepackage{relsize}
\usepackage{hyperref}      
\hypersetup{
     colorlinks=true, 
     breaklinks=true, 
     urlcolor= blue,  
     linkcolor= black, 
     linktoc	=all,	
     citecolor=red,	
     filecolor=black,	
     bookmarksopen=true,            
     pdftitle={}, 
     pdfauthor={Samuel Le Fourn},     
}

\numberwithin{equation}{section}
\allowdisplaybreaks

\theoremstyle{plain}
\newtheorem{thm}{Theorem}[section]
\newtheorem{prop}[thm]{Proposition}
\newtheorem{cor}[thm]{Corollary}

\newtheorem{lem}[thm]{Lemma}

\theoremstyle{definition}
\newtheorem{defi}[thm]{Definition}

\theoremstyle{remark}
 \newtheorem{rem}[thm]{Remark}

\renewcommand{\a}{\alpha}
\renewcommand{\b}{\beta}

\renewcommand{\d}{\delta}

\newcommand{\z}{\zeta}

\newcommand{\gp}{{\mathfrak{p}}}
\newcommand{\gP}{{\mathfrak{P}}}

\newcommand{\Bcal}{{\mathcal B}}

\newcommand{\Ical}{{\mathcal I}}

\newcommand{\Ocal}{{\mathcal O}}

\newcommand{\Xcal}{{\mathcal X}}
\newcommand{\Ycal}{{\mathcal Y}}

\newcommand{\Z}{{\mathbb{Z}}}
\newcommand{\Q}{{\mathbb{Q}}}
\newcommand{\R}{{\mathbb{R}}}

\renewcommand{\P}{\mathbb{P}}

\newcommand{\ord}{\operatorname{ord}}

\newcommand{\Supp}{\operatorname{Supp}}

\newcommand{\quot}[2]                                    
{\raisebox{.6ex}{\newline$#1$}\!/\!\raisebox{-.6ex}{$#2$}}

\newcommand{\Kb}{\overline{K}}

\title{Tubular approaches to Baker's method for curves and varieties}
\author{Samuel Le Fourn\footnote{Supported  by the European Union’s Horizon 2020 research and programme under the Marie Sklodowska-Curie grant agreement No 793646, titled LowDegModCurve.} \\ University of Warwick \footnote{samuel.le-fourn@warwick.ac.uk}}
\date\today

\begin{document}
\maketitle

\begin{abstract}
	Baker's method, relying on estimates on linear forms in logarithms of algebraic numbers, allows one to prove in several situations the effective finiteness of integral points on varieties. In this article, we give a generalisation of results of Levin regarding Baker's method for varieties, and explain how, quite surprisingly, it mixes (under additional hypotheses) with Runge's method to improve some known estimates in the case of curves by bypassing (or more generally reducing) the need for linear forms in $p$-adic logarithms. 
\end{abstract}

\section{Introduction}

	One of the main concerns of number theory is solving polynomial equations in integers, which amounts to determining the integral points on the variety defined by those equations. For a smooth projective curve over a number field, Siegel's theorem says that there are generally only finitely many integral points on this curve, but this result is as of yet deeply ineffective in that it does not provide us with any way to actually determine this set of integral points.
	
	We focus here on Baker's method (and to a lesser extent Runge's method), which are both \textit{effective}:  when applicable, they give a bound on the height of the integral points considered.
	Our work is based on Bilu's conceptual approach \cite{Bilu95} for curves, and its generalisation to higher-dimensional varieties by \cite{Levin14}. It is also heavily inspired (sometimes implicitly) by a previous article \cite{LeFourn3}, dealing with Runge's method. Before stating the main results, let us give some notations. 
	
	$K$ is a fixed number field, $L$ is a finite extension of $K$ with set of places $M_L$ divided into its archimedean places $M_L^\infty$ and its finite places $M_L^f$, and $S$ a finite set of places of $L$  containing $M_L^{\infty}$ (the pair $(L,S)$ will be allowed to change). The set of $S$-integers of $L$ is denoted by $\Ocal_{L,S}$ and the regulator of $\Ocal_{L,S}^*$ by $R_S$. We also will denote by $P_S$ the largest norm of an ideal coming from a finite place of $S$ (equal to 1 if $S = M_L^\infty$).
	
	The notion of integral point on a variety will be precisely defined (model-theoretically) in paragraph \ref{subsecnotations}, but the result is compatible with all reasonable definitions, e.g. the one in (\cite{Vojtadiophapp}, section I.4). The set $(X \backslash Y)(\Ocal_{L,S})$ will thus be the set of $S$-integral points of the variety $X \backslash Y$, where $Y$ is a closed subset of $X$.
	
	\begin{thm}
	\label{thmmainBaker}
	
		Let $X$ be a smooth projective variety over $K$ and $D_1, \cdots, D_n$ be ample effective divisors on $X$, $D = \bigcup_{i=1}^n D_i$, and $h_D$ a choice of global logarithmic height relative to $D$.
		
		The number $m_\Bcal$ (assumed to exist) is the smallest integer such that for any set $I \subset \{1, \cdots,n\}$ with $|I| = m_\Bcal$, the intersection $T_I = \bigcap_{i \in I} \operatorname{Supp} (D_i) (\overline{K})$ is finite. For any point $P$ in the finite set 
		\[
		T := \bigcup_{|I|=m_\Bcal} T_I,
		\]
		assume 
		
		$(H)_P$ : \enquote{there exists $\phi_P \in K(X)$ whose support is included in $\Supp D \backslash P$ (and $\phi_P(P)=1$ for simplicity)}. Such a function will be fixed in the following.
		
		Let $Y$ be a closed subset of $X$. The number $m_Y$ (assumed to exist) is the smallest integer such that for any set $I \subset \{1, \cdots,n\}$ with $|I| > m_Y$,  $\bigcap_{i \in I} \operatorname{Supp} (D_i)  \subset Y$).
		
		Then, there exists an effectively computable constants $C>0$ and an explicit function $C_1(d,s)$ such that for any pair $(L,S)$ where $[L:\Q]=d$ and $s=|S|$ satisfying 
		\begin{equation}
		\label{eqtubBakerhyp}
		(m_\Bcal-1) |M_L^{\infty}| + m_Y |S \backslash M_L^{\infty}| < n,
		\end{equation}
		for every $Q \in (X \backslash D)(\Ocal_{L,S}) \cap (X \backslash Y) (\Ocal_L)$,
		\begin{equation}
		\label{eqboundBaker}
		h_D(Q) \leq C \cdot C_1(d,s)  h_L R_S \log^* (h_L R_S),
		\end{equation}
		where $\log^*(x) = \max(\log x,1)$, unless 
		\[
		Q \in Z:= \bigcup_{P \in T} Z_{\phi_P}, \quad Z_{\phi_P} := \overline{\{ Q \in (X \backslash \Supp \phi_P)(\overline{K}), \, \phi_P(Q) = 1 \} },
		\]
		this set being an effective strict closed subset of $X$ independent of $(L,S)$.
	\end{thm}

	In the case of curves, one obtains the following corollary.
	
	\begin{cor}
	\label{corBaker}
	
		Let $C$ be a smooth projective curve over $K$ and $\phi \in K(C)$ nonconstant.
		
		Assume that every pole $P$ of $\phi$ satisfies $(H)_P$ (and to simplify, is defined over $K$). Then, there are a constant $C>0$ and an explicit function $C_1(d,s)$ such that for any pair $(L,S)$ satisfying 
		\begin{equation}
		\label{eqalmostunifconditioncurves}
		|S \backslash M_L^{\infty}| <n,
		\end{equation}
		any point $Q \in C(L)$ such that $\phi(Q) \in \Ocal_{L,S}$ satisfies 
		\[
		h(\phi(Q)) \leq C \cdot C_1(d,s)  h_L R_S\log^* (h_L R_S).
		\]
	\end{cor}

\begin{rem}
	Let us make some comments about those results.
	\begin{itemize}
		\item Under the assumption \eqref{eqtubBakerhyp} that $S$ does not have too many finite places (exactly translated by \eqref{eqalmostunifconditioncurves} for curves), one obtains a bound on the height which only grows in $R_S^{1+\varepsilon}$. In particular, there is no linear dependence on $P_S$ (which would come from estimates of linear forms in $p$-adic logarithms in a straightforward application of Baker's method), but rather in $\log P_S$, implicit in $R_S$, which might prove useful for some applications. 
		
		\item The set $Z$ can actually be made smaller: as done in \cite{Levin14}, we can replace each $Z_{\phi_P}$ by the intersection of \emph{all} $Z_{\phi_P}$ where $\phi_P$ runs through all functions satisfying the hypotheses of $(H)_P$. 
		
		\item The choice of subset $M_L^\infty$ of $S$ is quite arbitrary: replacing $M_L^\infty$ by a fixed $S' \supset M_L^{\infty}$ in each of its occurrences in Theorem \ref{thmmainBaker} and Corollary \ref{corBaker} (in particular the point $Q$ is now in $(X \backslash D)(\Ocal_{L,S}) \cap (X \backslash Y) (\Ocal_{L,S'})$) and considering sets of places $S \supset S'$, we obtain estimates of the shape \eqref{eqboundBaker} with an additional factor $P_{S'} \log^* P_{S'}$. With this consideration in mind, Theorem \ref{thmmainBaker} applied for $Y=D$ and $S=S'$ retrieves Levin's result (\cite{Levin14}, Theorem 1), and for smaller $S'$ (hence more hypotheses) improves the quantitative estimates it implicitly gave. 
		
		\item For general $Y$, Theorem \ref{thmmainBaker} improves qualitatively (when the hypotheses $(H)_P$ hold) a previous result based on Runge's method (\cite{LeFourn3}, Theorem 5.1 and Remark 5.2$(b)$), as condition \eqref{eqtubBakerhyp} is generally weaker than the tubular Runge condition defined there (the choice of set $(X \backslash Y) (\Ocal_{L,S'})$ is inspired by the notion of tubular neighborhood defined in that article, see section 3 there). Indeed, the $m=m_\emptyset$ in the statement of (\cite{LeFourn3}, Remark 5.2) is not $m_\Bcal$, and in general one only knows that $m_\Bcal \leq m+1$. If $m_\Bcal=m+1$, the original Runge's method can be applied as in \cite{Levin08} and gives better (and uniform) estimates than \cite{Levin14} (and the same holds for the tubular variants we propose), but if $m_\Bcal \leq m$ (the most likely situation), the condition \eqref{eqtubBakerhyp} is indeed more easily satisfied than for Runge's method.
		
	\item The words \enquote{effectively computable}  for $C$ deserve to be made more precise. One requires to know an embedding of $X$ in a projective space $\P^N_K$, explicit equations and formulas for $X$, the $D_i$, $D$, $h_D$, the points of $T$ and expressions of $\phi_P$ relative to this embedding. With this data, the effectivity boils down to an effective Nullstellensatz such as e.g. (\cite{MasserWustholz83}, Theorem IV). Now, the functions $C_1(d,s)$ (as well as $R_S \log^*(R_S)$) are coming from the theory of linear forms in logarithms, and as such completely explicit. The addition of $h_L$ is a technicality due to the necessity of slightly increasing the ring of units $\Ocal_{L,S}^*$ upon which to apply Baker's estimates, and can probably often be removed in special cases. 
	
	\item Unless we are in the case of curves, $m_\Bcal>1$ and then \eqref{eqtubBakerhyp} bounds $d$ and $s$ in terms of $n$, which allows us to replace $C_1(d,s)$ by an explicit functions $C_1(n)$.
	
	\end{itemize}
\end{rem}

	As an illustration of the effectivity of the method, we prove the following result on the $S$-unit equation: fix $L$ be a number field of degree $d$, $S \subset M_L^\infty$ a set of places of $L$ of cardinality $s$ and $\a,\b \in L^*$. We consider the $S$-unit equation
	\begin{equation}
\label{eqSunit}
\alpha x + \beta y = 1, \quad x,y \in \Ocal_{L,S}^*.
\end{equation}
	
	\begin{thm}
		\label{thmSunits}
		
		Let $L,S,\a,\b$ as above.
	
		\begin{itemize}
			\item If $S$ contains at most two finite places, all solutions of \eqref{eqSunit} satisfy
			\[
			\max(h(x),h(y)) \leq c(d,s) R_S \log^*(R_S) H,
			\] 
			where $H = \max(h(\a),h(\beta),1,\pi/d)$ and $c(d,s)$ is the constant defined as $c_{26}(s,d)$ in formula (30) of \cite{GyoryYu06}.
			
			\item For any set of places $S$, all solutions of \eqref{eqSunit} satisfy
			\[
			\max(h(x),h(y)) \leq c'(d,s) P'_S R_S (1+\log^*(R_S)/\log^* P'_S) H,
			\] 
			where $c'(d,s)=c_1(s,d)$ from Theorem 1 of \cite{GyoryYu06}, and $P'_S$ the third largest value of the norms of ideals coming from finite places of $S$. 
		\end{itemize} 
	\end{thm}

	This result provides an improvement on known bounds for solutions of the $S$-unit equations. More precisely, its dependence on $P'_S$ (instead of $P_S$ the largest norm of ideal coming from a place of $S$) becomes particularly interesting when there are at most two places of $S$ of large relative norm, and by construction it improves Theorem 1 of \cite{GyoryYu06}. One can also remark that such an estimate is likely to be close to optimality in terms of dependence on the primes in $S$, as replacing $R_S \log^*(R_S)$ by $o(R_S)$ in the first bound would imply that there are only finitely many Mersenne primes.
	
	On another hand, Theorem 4.1.7 of \cite{EvertseGyory15}, based on slightly different Baker-type estimates, has a better dependence on $s$ and $d$. It is possible to combine the strategy of proof of the latter theorem with our own to obtain an improvement of both results, essentially replacing again $P_S$ by $P'_S$. This is achieved in a recent preprint of Györy  \cite{Gyory19}.
	
	After proving Theorem \ref{thmmainBaker} and  Corollary \ref{corBaker} in section 3 (section 2 gathering the necessary reminders and tools for the proof), we prove Theorem \ref{thmSunits} in section 4. This application is heavily based on computations undertaken in \cite{GyoryYu06}, hence we have chosen to refer to it whenever possible, and focus on pointing out where the improvements come from our approach. 

\section*{Acknowledgements}

	I wish to thank K\'{a}lm\'{a}n Györy for his insightful comments on this paper and his remarks regarding the existing results on $S$-unit equations and their applications.

	\section{Reminders on Baker's theory and local heights}
	
		For any place $w$ of $L$, the norm $|\cdot|_w$ associated to $w$ is normalised to extend the norm on $\Q$ defined by $v_0$ below $w$, where $|\cdot|_\infty$ is the usual norm on $\Q$ and for every prime $p$ and nonzero fraction $a/b$,
		\[
		\left| \frac{a}{b} \right|_p = p^{\ord_p b - \ord_p a}.
		\]
		We also define $n_w = [L_w : \Q_{v_0}]$ the local degree of $L$ at $w$.
		
		In all discussions below, $X$ is a fixed projective smooth algebraic variety over the number field $K$ and closed subset of $X$ will mean a closed algebraic $K$-subvariety of $X$.
		
		Regarding the integrality, we choose a model-theoretic definition as follows. Assume $\Xcal$ is a proper model of $X$ over $\Ocal_K$, fixed until the end of this article. For every closed subset $Y$ of $X$, denote by $\Ycal$ the Zariski closure of $Y$ in $\Xcal$. The set of integral points $(X \backslash Y) (\Ocal_{L,S})$ will then implicitly denote the set of points $P \in X(L)$ whose reduction in $\Xcal_v(\kappa(w))$ for a place $w$ of $M_L \backslash S$ above $v \in M_K$ (well-defined by the valuative criterion of properness) never belongs to $\Ycal$.
	
		\subsection{\texorpdfstring{$M_K$-constants and $M_K$-bounded functions}{MK-constants and bounded functions}}
		\label{subsecnotations}
		
		The arguments below will be much simpler to present with the formalism of $M_K$-constants and $M_K$-functions briefly recalled here. 
		
		\begin{defi}
				\hspace*{\fill}
				
				\begin{enumerate}
			
				\item[$\bullet$] An \textit{$M_K$-constant} is a family $(c_v)_{v \in M_K}$ of nonnegative real numbers, all but finitely many of them being zero. 
				
				\item[$\bullet$] An \textit{$M_K$-function} $f$ (on $X$) is a function defined on a subset $E$ of $X(\Kb) \times M_{\Kb}$ with real values (typically, a local height function as below). Equivalently, it is defined as a function on a subset of $\bigsqcup_{K \subset L \subset \overline{K}} X(L) \times M_L$, consistently in the sense that if $f$ is defined at $(P,w)$ with $P \in X(L)$ and $w \in M_{L}$, then it is defined at $(P,w')$ with $w'|w \in M_{L'}$ for any extension $L'$ of $L$, and $f(P,w) = f(P,w')$.

				
				\item[$\bullet$] An $M_K$-function $f : E \rightarrow \R$ is \textit{$M_K$-bounded} if there exists an $M_K$-constant $(c_v)_{v \in M_K}$ for which for all $(P,w) \in E$, 
				\[
				|f(P,w)| \leq c_v \quad (w|v).
				\]
				The notation $O_{M_K}(1)$ will be used for an $M_K$-bounded function depending on the context (in particular, its domain $E$ will often be implicit but obvious).

				\item[$\bullet$]	Two $M_K$-functions $f,g : E \rightarrow \R$ are \textit{$M_K$-proportional} when there is an absolute constant $C>0$ and a $M_K$-constant $(c_v)_{v \in M_K}$ for which for all $(P,w) \in E$, 
				\[
				\frac{1}{C} |f(P,w)| -c_v \leq |g(P,w)| \leq C |f(P,w)| + c_v \quad (w|v).
				\]

				\item[$\bullet$] Two functions $f,g$ defined on an open subset $O$ of $X(\Kb)$ (typically, global height functions) are \textit{proportional} if there are absolute constants $C_1,C_2>0$ such that for every $P \in O$ : 
				\[
				\frac{1}{C_1} f(P) - C_2 \leq g(P) \leq C_1 f(P) + C_2.
				\]
%
			\end{enumerate}
			\end{defi}
		
		\subsection{Local heights associated to closed subsets}
		
		We will now define explicitly local height functions relative to closed subsets of a projective variety $X$.
		
		\begin{enumerate}

				\item[$\bullet$] For any point $P \in \P^N(L)$, one denotes by $x_P=(x_{P,0}, \cdots, x_{P,n}) \in L^{n+1}$ a choice of coordinates representing $P$ and $\|x_P\|_w  = \max_i |x_{P,i}|_w$.
				
				\item[$\bullet$] For a polynomial $g \in L[X_0, \cdots, X_N]$ and $w \in M_L$, the norm $\|g\|_w$ is the maximum norm of its coefficients for $| \cdot|_w$.
				
				\item[$\bullet$] Given a closed subset $Y$ of $\P^N_K$ and homogeneous polynomials $g_1, \cdots, g_m \in K[X_0, \cdots, X_N]$ generating the ideal of definition of $Y$, for any $w \in M_L$ and any $P \in (\P^N \backslash Y)(L)$, one defines explicitly a choice of local height of $P$ at $Y$ for $w$ by 
	
	\begin{equation}
	\label{eqdeflocheight}
	h_{Y,w} (P) := - \min_i \log \frac{|g_i(x_P)|_w}{\|g_i\|_w \| x_P\|_w^{\deg g_i}},
	\end{equation}
	and the global height by
	\[
	h_{Y}(P):= \frac{1}{[L:\P]} \sum_{w \in M_L} n_w \cdot h_{Y,w}(P).
	\]
	With this normalisation, for any $w \in M_L^f$ and $P \in (\P^N \backslash Y)(L)$, $h_{Y,w} (P) \geq 0$ and it is positive if and only if $P$ reduces in $Y$ modulo $w$.
\end{enumerate}
	
	Let us now sum up the main properties of those functions that we will need.
	
	\begin{prop}[Local heights]
		\hspace*{\fill}
		
		\label{proplocalheights}
		
		Let $X$ be a smooth projective variety over $K$, with an implicit embedding in a $\P^n_K$ and fixed choices of local heights as in \eqref{eqdeflocheight} for all closed subsets considered below.
		
		$(a)$ For any closed subsets $Y,Y'$ of $X$ the functions $h_{Y \cap Y',w}$ and $\min(h_{Y,w},h_{Y,w'})$ are $M_K$-proportional on $(X \backslash (Y \cup Y'))(\Kb) \times M_{\Kb}$.
		
		$(b)$ For a disjoint union $Y \sqcup Y'$ of closed subsets of $X$, one has 
		\[
		h_{Y,w}(P) + h_{Y',w}(P) = h_{Y \sqcup Y',w}(P) + O_{M_K}(1)
		\]
	on $(X \backslash (Y \cup Y'))(\Kb) \times M_{\Kb}$.
		
		$(c)$ For $Y \subset Y'$ closed subsets, one has 
		\[
		h_{Y,w}(P) \leq h_{Y',w}(P) + O_{M_K}(1)
		\]
	on $(X \backslash Y') (\Kb) \times M_{\Kb}$.
		
		$(d)$ If $\phi : X' \rightarrow X$ is a morphism of projective varieties, the functions $(P,w) \mapsto h_{Y,w}(\phi(P))$ (resp. $h_{\phi^{-1}(Y),w}(P)$) are $M_K$-proportional on $(X' \backslash \phi^{-1}((Y))(\Kb) \times M_{\Kb}$.
		
		$(e)$ For any closed subset $Y$ of $X$, the function $(P,w) \mapsto h_{Y,w}(P)$ is $M_K$-bounded on the set of pairs  satisfying $P \in (X \backslash Y) (\Ocal_{L,w})$ (independently of the number field $L$).
		
		$(f)$ For any effective divisor $D$ on $X$ and any function $\phi \in K(X)$ with support of poles included in $\Supp D$, the function $(P,w) \mapsto |\phi(P)|_w$ is $M_K$-bounded on the set of pairs $(P,w) \in X(L) \times M_L^f$ satisfying $P \in (X \backslash D)(\Ocal_{L,w})$ (independently of the number field $L$).

		$(g)$ If $D$ and $D'$ are two ample divisors on $X$, for any two choices of global heights $h_D$ and $h_{D'}$, they are proportional on $(X \backslash \Supp (D \cup D')) (\Kb)$.
				
		Furthermore, all the implied $M_K$-constants and constants are effective .
	\end{prop}
	
	\begin{proof}
		This proposition is mostly a reformulation of results of \cite{Silverman87} already quoted in \cite{Levin14}. First, \eqref{eqdeflocheight} indeed defines local heights associated to closed subsets by (\cite{Silverman87}, Proposition 2.4) so most of the proposition is contained in (\cite{Silverman87}, Theorem 2.1). Let us point out the slight differences and explain how it is effective. In that article, local height functions are more precisely defined by their ideal sheaves, whereas we consider closed subsets hence reduced closed subschemes. Now, if two ideal sheaves $\Ical$ and $\Ical'$ have the same support, their local height functions are $M_K$-proportional. More concretely, let us fix $Y \subset Y'$ closed subsets of $X \subset \P^N_K$ and two systems of homogeneous generators $g_1, \cdots, g_m \in K[X_0, \cdots, X_N]$ and $h_1, \cdots, h_p$ of ideal sheaves with respective supports $Y$ and $Y'$ in $\P^N_K$. After multiplying by a suitable $n \geq 1$, one can assume all those polynomials' coefficients belong to $\Ocal_K$, and such an $n$ can be made effective in terms of the $\| g_i\|_v$ and $\|h_j\|_v$ for $v \in M_K^f$. Now, an effective Nullstellensatz (e.g. \cite{MasserWustholz83} applied to multiples of those generators), translated in the projective case, will give relations 
		\[
		a g_i ^k = \sum_{j=1}^p f_{i,j} h_j
		\]
		with $a \in \Ocal_K$ nonzero, all the $f_{i,j}$ with coefficients in $\Ocal_K$ and bounded $\|f_{i,j}\|_v$ and $|a|_v$ in terms of the norms of the polynomials for all $v \in M_K^\infty$. Furthermore, the power $k$ is effectively bounded in terms of $[K : \Q]$,$N$ and the degrees of the polynomials. This will clearly give an effective inequality
		\[
		h_{Y,w}(P) \leq k \cdot h_{Y',w}(P) + O_{M_K}(1).
		\]
		for all $(P,w) \in (X \backslash Y')(\Kb) \times M_{\Kb}$, and this argument works for parts $(a)$ to $(d)$ of the Proposition (the inequality in $(c)$ without a factor $k$ coming from the fact that we can extend generators of an ideal sheaf for $Y'$ to an ideal sheaf for $Y$). 
		
		The only parts remaining to be proven are now $(e)$, $(f)$ and $(g)$. Part $(e)$, essentially saying that local height functions detect integral points up to some $M_K$-bounded error, is classical (see \cite{Vojtadiophapp}, Proposition 1.4.7) and in fact automatic for the exact definition given in \eqref{eqdeflocheight}. Part $(f)$ then comes from $(e)$ and Lemma 11 of \cite{Levin14}. Finally, part $(g)$ is a classical result on heights (e.g. \cite{LangFundamentals}, Chapter IV, Proposition 5.4), and there also, the constants implied can be made effective.
	\end{proof}
	
	\subsection{Baker's theory of linear forms in logarithms}
	
	Let us now give our second main tool: estimates from Baker's theory in a special form sufficient for our purposes.

	For any place $w$ of $M_L$, $N(w)$ is defined to be 2 if $w$ is archimedean and the norm of the associated prime ideal otherwise.
		
	\begin{prop}
		
		\label{propBaker}
		
		Define $\log^*(x) = \max(\log(x),1)$ for $x >0$.
		
		Let $d=[L: \Q]$ and $s = |S|$. Recall that $P_S$ is the largest norm of an ideal coming from a finite place of $S$ (equal to 1 if there is no such place). There is an effectively computable function $C(d,s)$ such that for any pair $(L,S)$, any $\alpha \in \Ocal_{L,S}^* \backslash \{1\}$ and any $w \in M_L$,
		\begin{equation}
		\label{eqBakertheory}
		\log |\alpha - 1|_w \geq - C(d,s) N(w) R_S \log^* h(\alpha).
		\end{equation}
		
		In terms of local heights, one can choose the local height $h_{1,w}(\alpha)$ to be $\max(- \log |\alpha - 1|_w,0)$, which gives us
		\[
		h_{1,w}(\alpha) \leq C(d,s) N(w) R_S \log^* h(\alpha).
		\]
		
	\end{prop}
	\begin{proof}
		This result, although natural when one knows estimates for linear forms in logarithms, is not often presented in this form, so the following proof will explain how one can get to such an expression with known results. First, let us assume $s \geq 2$ (for $s=1$, the result is trivial). By Lemma 1 of \cite{BugeaudGyory96} ($\log h$ there is our logarithmic height here), one can choose a family of fundamental units $\varepsilon_1, \cdots, \varepsilon_{s-1}$ of $\Ocal_{L,S}^*$ such that 
				\[
				\prod_{i=1}^{s-1} h(\varepsilon_i) \leq c_1(s) R_S,
				\]
				where $c_1 (s) = ((s-1)!)^2 /(2^{s-1} d^{s-2})$.
				
		Now, by Theorem 4.2.1 of \cite{EvertseGyory15} applied to $\Gamma = \Ocal_{L,S}^*$ and this system of fundamental units, we obtain (taking into account our normalisation of $|\cdot|_w$) the bound
		\[
		\log |\alpha - 1|_w > - \frac{c_2(d,s)}{n_w} c_1(s) R_S \frac{N(w)}{\log N(w)} \log^* (N(w) h^*(\alpha)), 
		\]
		with $c_2(d,s) = c_8$ in the reference. We can put all effective constants in $C(d,s)$, and roughly bound $\frac{N(w)}{\log N(w)} \log^* (N(w) h^*(\alpha))$ by $2 N(w) \log ^* h(\a)$ to obtain the stated result.
	\end{proof}

	\section{Proof of the main theorem}
	
	We now have all the tools to prove the theorem. We keep the notations from its statement, and assume we have an embedding $X \subset \P^N_K$ from which all local heights considered below are defined. Recall that $s=|S|$, $d=[L:\Q]$ and $n_w$ is the local degree of $L$ at $w$. The constants $c_i$ below are absolute and can be made effective.
	
	First, let us notice that for every point $P \in T$, as the support of $\phi_P$ is in $D$, by Proposition \ref{proplocalheights} $(f)$ applied to $\phi_P$ and $\phi_P^{-1}$,  there is an absolute positive integer $m$ (independent on the choice of $(L,S)$ and $P \in T$) such that for every $Q \in (X \backslash D) (\Ocal_{L,S})$, one has $m \phi_P(Q) \in \Ocal_{L,S}$ and $m \phi_P(Q)^{-1} \in \Ocal_{L,S}$. Defining $S_m$ the set of primes of $L$ dividing $m$, one thus has $\phi_P(Q) \in \Ocal_{L,S \cup S_m}^*$ for all $Q \in (X \backslash D)(\Ocal_{L,S})$. 
	
	By Proposition \ref{proplocalheights} $(e)$, the map $(Q,w) \mapsto h_{D_i,w}(Q)$ is $M_K$-bounded on pairs $(Q,w)$ with $w \in M_L \backslash S$ and $Q \in (X \backslash D) (\Ocal_{L,S})$, and $(Q,w) \mapsto h_{Y,w}(Q)$ is $M_K$-bounded on pairs $(Q,w)$ with $w \in M_L \backslash M_L^{\infty}$ and $Q \in (X \backslash Y) (\Ocal_L)$.
	
	Let us assume now that $Q \in (X \backslash D)(\Ocal_{L,S}) \cap (X \backslash Y) (\Ocal_L)$. The previous paragraphs imply that for every $i \in \{1, \cdots, n\}$ 
	\[
	h_{D_i}(Q) = \frac{1}{[L:\Q]} \sum_{w \in S} n_w h_{D_i,w}(Q) + O(1)
	\]
	where $O(1)$ is absolutely (and effectively bounded) on the set of such points $Q$ (even if $(L,S)$ is allowed to change).
	
	Thus, for all $i \in \{1, \cdots, n\}$, there is $w \in S$ such that 
	\[
	h_{D_i,w} (Q) \geq  \frac{n_w}{[L:\Q]} h_{D_i,w} (Q)\geq \frac{1}{s} h_{D_i} (Q) + O(1).
	\]
	 
	After choosing for every $i \in \{ 1, \cdots, n\}$ such a $w \in S$, we obtain a function $\{1, \cdots, n\} \rightarrow S$. Now, if the fiber above a place $w \in S \backslash M_L^{\infty}$ was a set $J$ with $|J| > m_Y$, by Proposition \ref{proplocalheights} $(a)$, one would obtain an absolute effective (computable) upper bound on the minimum of such $h_{D_j,w}(Q)$, therefore on $h_{D_i}(Q)/s$ and $h_D(Q)/s$ by Proposition \ref{proplocalheights}
 $(g)$.
	
	We can thus assume from now on that this is not the case. Therefore, the fibers of this function are of cardinality at most $m_Y$ above $S \backslash M_L^{\infty}$. Consequently, by hypothesis \eqref{eqtubBakerhyp}, one of the fibers above $M_L^\infty$, defined as $I$, has to be of cardinality at least $m_\Bcal$ (if it's more, we extract a subset of cardinality $m_\Bcal$), which gives $w \in M_L$ such that
	\[
	\min_{i \in I} h_{D_i,w} (Q) \geq \frac{1}{s} \min_{i \in I} h_{D_i} (Q) \geq \frac{c_1}{s} h_D(Q) + O(1).
	\]
	Now, by Proposition \ref{proplocalheights} $(a)$ again and construction of $T$ (if $T=\emptyset$, we directly obtain an absolute $M_K$-constant bound), there exists $P \in T$ such that 
	\begin{equation}
	\label{eqheightscomparison}
	h_{P,w}(Q) \geq \frac{c_2}{s} h_D(Q) + O_{M_K}(1) \geq \frac{c_3}{s} h (\phi_P(Q)) + O_{M_K}(1),
	\end{equation}
	using Proposition \ref{proplocalheights} $(g)$, with absolute effective constants $c_1,c_2,c_3>0$.
	
	By Proposition \ref{proplocalheights} $(d)$, as $\phi_P(P)=1$, if $\phi_P$ can be evaluated at $Q$ and $\phi_P(Q) \neq 1$, 
	\[
	h_{1,w} (\phi_P(Q)) \geq \frac{c_4}{s} h(\phi_P(Q)),
	\]
	for $c_4>0$ absolute effective, computable in terms of the initial data of embeddings and equations.

	Now, applying Proposition \ref{propBaker} to $\Ocal_{L,S \cup S_m}^*$ (here, $w$ is archimedean), this inequality implies the  bound
	\begin{equation}
	\label{eqintermformulaBaker}
	h(\phi_P(Q)) \leq 2 s \cdot C(d,s+|S_m|)/c_4 \cdot R_{S \cup S_m} \cdot \log^* h(\phi_P(Q))
	\end{equation}
	
	By formula (1.8.3) of \cite{EvertseGyory15}, one has 
	\begin{equation}
	\label{eqboundsregulator}
	R_{S \cup S_m} \leq h_L R_S \prod_{\gP \in S_n \backslash S} \log N(\gP)) \leq h_L R_S \prod_{p|m} e^{d/e \log p} \leq h_L R_S m^{d/e}
	\end{equation}
	after optimising the products of logarithms.
	
	Combining \eqref{eqintermformulaBaker} and \eqref{eqboundsregulator}, we obtain an affine bound of the shape \eqref{eqboundBaker} for $h(\phi_P(Q))$, hence on $h_{P,v} (Q)$ by Proposition \ref{proplocalheights} $(d)$ applied the other way, which finally gives a bound on $h_D(Q)$  by \eqref{eqheightscomparison} (there is a constant term which we can absorb in the linear one as it is effectively boundable).
	
	To apply the same method with $S' \supset M_L^\infty$ instead of $M_L^\infty$, we apply exactly the same process under the condition 
	\[
	(m_\Bcal-1) |S'| + m_Y |S \backslash S'| <n, 
	\]
	to obtain for our point $Q$ a place $w \in S'$ and a point $P$ such that $h_{P,w}(Q) \geq c_3/s h(\phi_P(Q)) + O(1)$, and then follow the end of the proof with the general estimate of Proposition \ref{propBaker} for places of $S'$ instead of the archimedean case (this is where the additional factor $P_{S'} \log^* P_{S'}$ comes from).
	
	\section{\texorpdfstring{Applications to the $S$-unit equation in the case of curves}{Applications to the S-unit equation in the case of curves}}

	In this section, we realise our method in the practical case of the $S$-unit equation \eqref{eqSunit}, to prove Theorem \ref{thmSunits}.
	
	This problem is related to finding the integral points of $(\P^1 \backslash \{0,1,\infty\})(\Ocal_{L,S})$ (up to taking into account the factors $\alpha,\beta$), this is the interpretation we will follow below to illustrate the main theorem.
	We follow closely the definitions and lemmas of \cite{GyoryYu06} (except their normalisations of norms), as our improvements intervene only at the beginning of the proof. As in that article, define 
	\[
	d = [L:\Q], \quad H = \max(h(\alpha),h(\beta),1,\pi/d), \quad s=|S|.
	\]
	For any $t \in L$, define $h_w (t) :=  h_{0,w} (t) = \log^+(1/|t|_w)$.
	
	For sake of symmetry of the exposition, we will do most computations with $ \a x  $ and $ \b y $, before coming back to $ H $. This means we deal with $h_w(P)$ for 
	
	\[ P \in E=\left\{ \a x,\b y,\frac{\b y}{\a x}\right\}. \]

	\begin{lem}
	\label{lemSunits}
		For any $x,y \in L$ with $\alpha x + \beta y = 1$: 
		
		\begin{itemize}
			\item For any place $w \in M_L$, at most one value of $h_w(P)$ for $P \in E$ can exceed $\d_w \log 2$, where $ \d_w =1 $ if $w$ is infinite, 0 otherwise. 
			
			\item The maximum module of the difference of logarithmic heights of any two of them amongst $h(x),h(y),h(\a x), h(\beta y),h(\frac{\b y}{\a x})$ is at most $3H$, and even $2H$ except for $|h(x)-h(y)|$.
			
			\item If $x,y \in \Ocal_{L,S}^*$ and $h=\max(h(x),h(y))$, we always have, for $P \in E$:
			\[ 
			\sum_{w \in S} \frac{n_w}{[L:\Q]}h_w(P) \geq h-3H.
			 \]
		\end{itemize}

	\end{lem}

\begin{proof}
	The first item is the translation of the fact that is $z+z'=1$ one of $z,z'$ has to have norm at least 1 if the norm is ultrametric, and at least $1/2$ if it is archimedean.
	
	The second item uses that for any nonzero algebraic numbers $z,z'$, $h(zz') \leq h(z)+h(z')$ and $h(z+z') \leq h(z) + h(z') + \log 2$. For example, we obtain $|h(x) - h(\a x)| \leq H$ and $|h(\a x)-h(\b y)| \leq \log 2$, and by symmetric role this leads to all other bounds on difference of heights, as $\log 2 \leq H$.
	
	For the third item, there are six cases to deal with (depending on which of $h(x)$ and $h(y)$ is the maximum, and the values of $P$). If $P=\a x$, we have 

\[ 	\sum_{w \in S} \frac{n_w}{[L:\Q]}h_w(P)  \geq \sum_{w \in S}  \frac{n_w}{[L:\Q]} \log^+ |1/(\a x)|_w \geq \sum_{w \in S}  \frac{n_w}{[L:\Q]} (\log^+ |1/ x|_w-\log^+ |\a|_w) \geq h(x)-H
\]
	because $x \in \Ocal_{L,S}^*$. In the same fashion, this term is also bounded from below by $h(P) - h(\a)$. By symmetric role, for $P=\beta y$, we obtain 
	\[ 	\sum_{w \in S} \frac{n_w}{[L:\Q]}h_w(P)  \geq \max(h(y),h(P))-H.
	\]
	Now, if $h=h(y)$ and $P=\a x$, we use that $h(\a x) \geq h(\b y) - \log 2$ to obtain the desired bound, and similarly for the symmetric case. There only remains the case $P=\frac{\b y}{\a x}$. Using the fact that $1/P = 1 + 1/(\b y)$, we obtain 
	\[
	\sum_{w \in S} \frac{n_w}{[L:\Q]}h_w(P) \geq h(P) - h(\beta) \geq h -3H
	\]
	by the second item.
\end{proof}
	
	\begin{proof}[Proof of Theorem \ref{thmSunits}]
		
		First, notice that if $s \leq 2$, Lemma 4.1 alone gives immediately than there is $P \in E$ such that $h_w(P) \leq \d_w \log 2$ for all $w \in S$, and elsewhere we have $h_w(P) = h_w(\a),h_w(\b)$ or $h(\b/\a)$ depending on the value of $P$, because $x$ and $y$ are $S$-units. Consequently, $h(P) \leq 2H + \log 2$, hence $h \leq 4H + \log 2$ in this case. We can now assume that $s \geq 3$.
		
	For the first part of Theorem \ref{thmSunits}, the assumption amounts to saying that \eqref{eqtubBakerhyp} holds in this case. By Lemma \ref{lemSunits}, for any choice of $P \in E$, there is $w \in S$ such that
	
	\begin{equation}
	\label{eqineqcleBaker}
	\frac{n_w}{[L:\Q]} h_{w}(P) \geq \frac{1}{|S|} (\max(h(x),h(y)) - 3 H),
	\end{equation}
	
	We want to fall back on a case where for one of the three choices of $P$, one can impose that $w \in M_L^\infty$. If that is not possible, by pigeonhole principle and our hypothesis on $S$, there is a finite place $w$ and two points $P,Q \in \{E \}$ distinct with 
	\[
	\frac{n_w}{[L:\Q]} \min(h_{w}(P),h_{w}(Q)) \geq \frac{1}{|S|}(\max(h(x),h(y)) -3 H).
	\]
	By the same lemma, we get $\max(h(x),h(y)) \leq 3H$. This bound will be readily checked to be smaller than the other case. 
	
	One can thus assume from now on that for some $w \in M_L^\infty$ and $P \in  E $, \eqref{eqineqcleBaker} holds.
	
	The only thing to do is then to get back to the situation of (\cite{GyoryYu06}, page 24) in all three cases, after which we will obtain the exact same bounds. 
	We fix a fundamental system $ \varepsilon_1, \cdots, \varepsilon_{s-1}$ of units of $ \Ocal_{L,S}^* $ with the properties of (\cite{GyoryYu06}, Lemma 2). 
	
	$ \bullet $ Assume first $P=\a x$, and write 
	\begin{equation}
	\label{eqfactory}
	y= \zeta \varepsilon_1 ^{b_1} \cdots \varepsilon_{s-1} ^{b_{s-1}},
	\end{equation}
	with $\zeta$ a root of unity in $L$ and $b_i \in \Z$ for all $ i $. By the arguments of \cite{BugeaudGyory96} (p. 76), we obtain that 
	\[ 
	B = \max(|b_1|, \cdots, |b_{s-1}|) \leq c_1(d,s) h(y)
	\]
	with 
	\[ 
	c_1(d,s) = \left\{ \begin{array}{lcr}
	 ((s-1)!)^2/(2^{s-3} \log 2) & \textrm{if} & d=1 \\
	((s-1)!)^2/2^{s-2}) \log(3d)^3 & \textrm{if} & d \geq 2.
	\end{array}
	\right.
	 \]
	 We set $\alpha_s = \zeta \beta$ and $b_s=1$ so that 
	 \[
	 |\a x|_w = |1 - \varepsilon_1 ^{b_1} \cdots \varepsilon_{s-1} ^{b_{s-1}} \a_s^{b_s}|_w.
	 \]
	 We set the $A_i$ and $A_s$ as in (\cite{GyoryYu06}, equation (31)) and can make the same assumption  (otherwise, we obtain a smaller bound). By (\cite{GyoryYu06}, Proposition 4 and Lemma 5), we thus obtain 
	 
\begin{equation}
\label{eqintermhw}
	 h_w(P) = - \log |\a x|_w < c_2(d,s) c_3(d,s) R_S H \log \left( \frac{c_1(d,s) h(y)}{\sqrt{2}H}\right)
\end{equation}
	 
	 (as we always have $s \geq 3$ here), with 
	 \[
	 c_2(d,s) = d^3 \log(ed) \min(1.451 (30\sqrt{2})^{s+4}(s+1)^{5.5},\pi 2^{6.5s+27}),
	 \]
	 \[
	 c_3(d,s) = e\sqrt{s-2} (((s-1)!)^2/(2^{s-2})) \pi^{s-2} \cdot  \left\{ \begin{array}{lcr}
	 8.5 & \textrm{if} & d=1 \\
	29 d \log d & \textrm{if} & d \geq 2,
	 \end{array}
	 \right.
	 \]
	 Let us now define $h=\max(h(x),h(y))$. We use inequality \eqref{eqineqcleBaker}, and replace $h(y)$ by $h$ in the right-hand side of \eqref{eqintermhw} to obtain 
	 \begin{equation}
	 \label{eqfinaleBaker}
	 \frac{h}{s} -H \leq \frac{h - 3H}{s} \leq \frac{n_w}{[L:\Q]} c_2(d,s) c_3(d,s) R_S H \log \left( \frac{c_1(d,s) h}{\sqrt{2}H}\right)
	 \end{equation}
	 and these are equivalent to the two inequalities used page 24 to obtain the result.
	 
	 $ \bullet $ Assume $P=\b y$. We apply the same argument by symmetry, replacing $ \a $ by $ \b $ and $ x $ by $ y $ everywhere, to finally obtain the same bound.
	 
	 $\bullet$ Assume $P=\frac{\b y}{\a x}$. We thus write 
	 \[
	 h_w(P) = - \log \left|\frac{\b y}{\a x} \right|_w = - \log \left|1 - \frac{1}{\a x}\right|_w.
	 \]
	 Write then 
	 \[
	 \frac{1}{x} = \zeta \varepsilon_1 ^{b_1} \cdots \varepsilon_{s-1} ^{b_{s-1}}, \quad \alpha_s = \frac{\z}{\a}.
	 \]
	 In the same fashion as above, we then get $ B \leq c_1(d,s) h(x)$, and the rest follows, with Lemma 4.1, in the exact same way until we obtain \eqref{eqfinaleBaker}.
	
	For the second part of Theorem 4.1, we can play the same game, by defining $S'$ the set of places $S$ deprived of its two prime ideals with largest norm. The same elimination work as before will then give $w \in S'$ and $P \in E$ satisfying \eqref{eqineqcleBaker}, and from there we can apply for $P = \a x$ the exact method of (\cite{GyoryYu06}, page 25) with a prime ideal $\gp$ from $S'$. This finally leads to the same estimates with $P'_S$ instead of $P_S$, using again Lemma 4.1. 
	\end{proof}

\end{document}